\documentclass[11pt]{article}
\usepackage[english]{babel}
\usepackage[utf8]{inputenc}
\usepackage{amsmath}
\usepackage{color,latexsym,amsthm,amsmath,amssymb,path,enumitem,url,bbm,subcaption,geometry,bbm}
\usepackage[backref]{hyperref}
\usepackage{wrapfig}
\usepackage[pdftex]{graphicx}
\usepackage{comment}
\usepackage[utf8]{inputenc}
\usepackage[font=footnotesize,skip=3pt]{caption}
\usepackage{keytheorems}
\usepackage{enumitem}

\usepackage[parfill]{parskip}
\usepackage{etoolbox}

\makeatletter

\patchcmd\deferred@thm@head
  {\addvspace{-\parskip}}
  {}
  {}{\typeout{\string\deferred@thm@head patch failed!}}

\makeatletter 

\newkeytheorem{example}[parent=section]
\newkeytheorem{theorem}[sibling=example]
\newkeytheorem{proposition}[sibling=example]
\newkeytheorem{lemma}[sibling=example]
\newkeytheorem{corollary}[sibling=example]
\newkeytheorem{claim}[sibling=example]
\newkeytheorem{observation}[sibling=example]
\newkeytheorem{conjecture}[sibling=example]

\theoremstyle{definition}
\newkeytheorem{definition}[sibling=example]
\newkeytheorem{remark}[sibling=example]
\newkeytheorem{open}[name={Open Question}, sibling=example]
\newkeytheorem{notation}[sibling=example]
\newkeytheorem{convention}[sibling=example]

\newcommand{\parag}[1]{\vspace{2mm}
\noindent{\bf #1} }

\newcommand{\RR}{\mathbb R}

\newcommand{\ZZ}{\mathbb Z}

\newcommand{\pts}{\mathcal{P}}
\newcommand{\curves}{\Gamma}

\newcommand{\vecx}{\mathrm{x}}

\newcommand{\Pf}{\mathrm{Pfaff}}

\newcommand*{\vecw}{\mathrm{w}}
\newcommand*{\vecv}{\mathrm{v}}

\title{Distinct Distances on Pfaffian Curves}

\author{Abhiram Natarajan\thanks{Warwick Mathematics Institute, University of Warwick, UK.
{\sl abhiram.natarajan@warwick.ac.uk}.}
\and 
Adam Sheffer\thanks{CUNY: Baruch College, NY, USA.
{\sl adamsh@gmail.com}.}
}

\date{}

\begin{document}
\maketitle

\begin{abstract}
We generalize Pach and de Zeeuw's bound for distinct distances between points on two curves, from algebraic curves to Pfaffian curves.
Pfaffian curves include those that can be defined by any combination of elementary functions, including exponential and logarithmic functions, rational and irrational powers, trigonometric functions and their inverses, integration, and more.
The bound remains $\Omega(\min\{m^{3/4}n^{3/4},m^2,n^2\})$, as obtained from the proximity technique of Solymosi and Zahl.
\end{abstract}

\section{Introduction}

In recent decades, there has been an increasing interest in generalizing central Discrete Geometry theorems. 
Results concerning incidences, distinct distances, and related topics often begin with lines in $\RR^2$ and are then extended to algebraic curves, algebraic varieties in $\RR^d$, semi-algebraic sets, other fields, and more. For a few examples among many, see \cite{AC24,BKT04,RSdZ16,SV08}. One particular effort has been to extend discrete geometry over $\RR$ to include sets more general than semi-algebraic sets, yet retaining certain convenient topological properties. These properties --- for example, having finitely many connected components and admitting stratifications --- are usually referred to as \emph{tame topological}. There is a growing body of work in discrete geometry which studies questions in this broader tame setting \cite{BR18,CGS20,CPS24,CS18,CST21}.

In this paper we focus on Pfaffian functions, which were introduced by Khovanski{\u\i} \cite{Khovanskii91}. Sets defined by Pfaffian functions, called \emph{Pfaffian sets}, are significantly more general than semi-algebraic sets, allowing the use of exponential and logarithmic functions, rational and irrational powers, trigonometric functions and their inverses, integration, and more.\footnote{The use of sines and cosines is allowed only on a bounded domain.} At the same time, Pfaffian sets retain key properties of semi-algebraic sets, such as well-behaved intersections and a bounded number of connected and irreducible components.

An ambitious long-term goal is to extend the known bounds for expanding polynomials \cite{RSdZ16,RSS16} and the Zarankiewicz problem \cite{FPSSZ17,TidorYu24} from the semi-algebraic case to the Pfaffian case. 
The current work provides a modest first step in this direction, proving a distinct distances bound between two Pfaffian curves.
This is a special case of the expanding polynomials bounds of \cite{RSS16,SolyZahl24}.
Two other works which explore questions in extremal combinatorics involving Pfaffian sets are \cite{Balsera23,LNV24}.

For those curious about the model-theoretic origins of tameness, the framework of \emph{o-minimal geometry} provides a natural setting for going beyond semi-algebraic sets while preserving their tame topological properties. The axioms of o-minimality guarantee that any collection of subsets of $\RR^d$ satisfying them --- such a collection is called an \emph{o-minimal structure} --- exhibits tame behavior. Pfaffian sets (zero loci of Pfaffian functions) and semi-Pfaffian sets (defined by inequalities of Pfaffian functions) belong to the Pfaffian structure, which is the smallest collection containing all semi-Pfaffian sets. Also, the Pfaffian structure is closed under the usual structure operations (finite unions, complements, Cartesian products, and projections. Speissegger \cite{Sp99} established that the Pfaffian structure is an o-minimal structure. By providing a rich class of sets with tame topological behavior, the Pfaffian structure plays a central role in o-minimal geometry. Readers interested in further details regarding o-minimality can consult \cite{Dries98}.

\parag{Distinct distances between curves.}
For points $p,q\in\RR^2$, we denote the Euclidean distance between $p$ and $q$ as $|pq|$.
For a set $\pts\subset\RR^2$, the number of \emph{distinct distances} spanned by pairs from $\pts$ is
\[ D(\pts)= \left|\left\{|pq|\ :\ p,q\in \pts \right\}\right|. \] 
Erd\H os \cite{Erdos46} discovered a set $\pts$ of $n$ points with\footnote{For a discussion of asymptotic notation such as $\Theta(\cdot), O(\cdot), \Omega(\cdot)$ and $\Omega_k(\cdot)$, see \cite[Appendix A]{Sheffer22}.} $D(\pts)=\Theta(n/\sqrt{\log n})$. 
He conjectured that every set $\pts$ of $n$ points satisfies $D(\pts)=\Omega(n/\sqrt{\log n})$.
After over 70 years, Guth and Katz \cite{GK15} almost completetly resolved Erd\H os's conjecture, by proving that $D(\pts)=\Omega(n/\log n)$.
The countless works on the subject led to a deep theory, with many other distinct distances problems that are still wide open.

For point sets $\pts_1,\pts_2\subset \RR^2$, we define
\[ D(\pts_1,\pts_2)= \left|\left\{|pq|\ :\ p\in \pts_1 \text{ and } q\in \pts_2 \right\}\right|. \] 
In other words, we have a bipartite distinct distances problem, considering only distances of pairs from $\pts_1\times\pts_2$.

Purdy studied the case where $\pts_1$ is a set of $m$ points on a line $\ell_1$ and $\pts_2$ is a set of $n$ points on a line $\ell_2$ (for example, see \cite[Section 5.5]{BMP05}).
He observed that, when $\ell_1$ and $\ell_2$ are parallel or orthogonal, we may have that $D(\pts_1,\pts_2)=\Theta(m+n)$.
Purdy conjectured that $D(\pts_1,\pts_2)$ is asymptotically larger when the lines are neither parallel nor orthogonal.
The current best bound for this problem, obtained by \cite{SSS13,SolyZahl24}, is
\[ D(\pts_1,\pts_2) = \Omega\left(\min\{m^{3/4}n^{3/4},m^2,n^2\}\right). \] 

The above result was generalized by Pach and de Zeeuw \cite{PachDZ17} to algebraic curves. 
Combined with a recent technique of Solymosi and Zahl \cite{SolyZahl24}, it leads to the following theorem.

\begin{theorem} \label{th:PachdeZeeuw}
Let $\gamma_1,\gamma_2\subset\RR^2$ be algebraic curves of degree at most $k$.
Let $\pts_1$ be a set of $m$ points on $\gamma_1$ and let $\pts_2$ be a set of $n$ points on $\gamma_2$.
If $\gamma_1$ and $\gamma_2$ are neither parallel lines, orthogonal lines, nor concentric circles, then
\[  D(\pts_1,\pts_2) = \Omega_k(\min\{m^{3/4}n^{3/4},m^2,n^2\}).\]
\end{theorem}

We generalize Theorem \ref{th:PachdeZeeuw} to Pfaffian curves. 
Since a rigorous definition of Pfaffian curves is somewhat long and technical, we postpone it to  Section \ref{sec:pfaffian}. 
As a first intuition, we may consider curves defined by any combination of elementary functions. For example, curves defined by $y=e^x$, by $x^\pi + y^\pi= 1$, and even by $\ln (\arctan y) = \int_{-\infty}^x \frac{e^t}{t}dt$. 
  
\begin{theorem}[store=thtwocurves]
\label{th:twocurves}
Let $Q$ be a Pfaffian chain of order $r$ and chain degree $\alpha$. 
Let $f_1$ and $f_2$ be bivariate Pfaffian functions of degree $\beta$ in $Q$. 
Let $\gamma_1,\gamma_2\subset \RR^2$ be the curves respectively defined by $f_1$ and $f_2$. 
Assume that $\gamma_1$ and $\gamma_2$ do not contain arcs of parallel lines, orthogonal lines, or concentric circles.
Then the number of distinct distances between any $m$ points on $\gamma_1$ and $n$ points on $\gamma_2$ is $\Omega_{\alpha,\beta,r}(\min\{m^{3/4}n^{3/4}, m^2, n^2\})$. 
\end{theorem}

By taking $\gamma_1=\gamma_2$, Theorem \ref{th:twocurves} also implies the following result.

\begin{corollary}\label{co:onecurve} 
Consider $f \in \Pf_2(\alpha, \beta, r) \setminus \{0\}$ defining a curve $\gamma = Z(f) \subseteq \RR^2$ that does not contain an arc of a line or a circle.
Then the number of distinct distances spanned by any set of $n$ points on $\gamma$ is $\Omega_{\alpha,\beta,r}(n^{3/2})$. 
\end{corollary}

\begin{wrapfigure}[10]{r}{5.2cm}
\vspace{-11pt}
\includegraphics[scale=0.25]{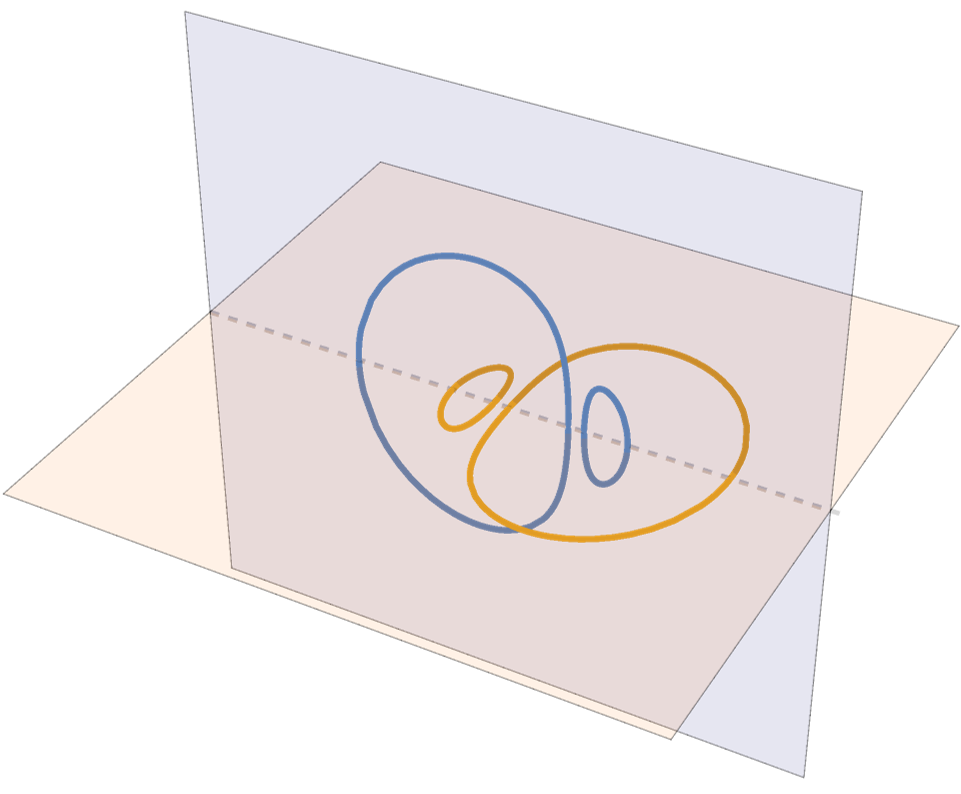}
\caption{A log-circles configuration}
\label{fig:log-circles}
\end{wrapfigure}
Theorem \ref{th:twocurves} shows that, when generalizing Theorem \ref{th:PachdeZeeuw} to Pfaffian curves, the exceptional cases remain to be parallel lines, orthogonal lines, and concentric circles.
Interestingly, this is no longer the case in $\RR^3$. 
As shown in \cite{ALPSV24}, there is a small number of distances between point sets on two \emph{log-circles} defined as  
 \[(x-B)^2+y^2 =D+A\cdot \ln(x)\quad \text{ and } \quad z=0, \]
\[ x^2 +z^2 =D+A\cdot \ln(x - B)\quad \text{ and } \quad y=0; \]
see Figure \ref{fig:log-circles}. This seems to hint that, in dimension $d\ge 3$, a proof for distinct distances between Pfaffian curves will require additional new ideas.

Characterizing the point sets in $\RR^2$ that span a sublinear number of distinct distances is a central open problem of Erd\H os. 
It is considered to be a highly difficult problem, for which not much is known.
For a survey of the problem and the few known related results, see \cite[Section 2]{Sheffer14}.
Theorem \ref{th:twocurves} immediately leads to the following related bound. 

\begin{corollary}
Let $\pts\subset\RR^2$ be a set of $n$ points with $D(\pts)=o(n)$. Then any constant complexity\footnote{See Section \ref{sec:pfaffian} for a definition of the complexity of a Pfaffian set.} Pfaffian curve that contains no arcs of lines or circles contains $o(n^{3/4})$ points of $\pts$.
\end{corollary}
   
\parag{The proof technique.} 
The proof of Theorem \ref{th:twocurves} begins by imitating the proof of Pach and de Zeeuw \cite{PachDZ17} for the algebraic case. 
In particular, we define a bipartite distance energy, study isometries of curves, and obtain an incidence problem (these terms are defined in detail below).

However, since Pfaffian curves behave differently than algebraic curves, several parts of the proof need to be replaced. 
For example, the projection of an algebraic variety to a lower dimensional space is well-behaved.\footnote{More precisely, the Zariski closure of the projection of an algebraic variety is algebraic.}
On the other hand, projections of semi-Pfaffian sets need not be semi-Pfaffian, or even coverable by a fixed number of semi-Pfaffian sets. Such an example is provided in \cite{Osgood16}.
Thus, we need to replace a step of the proof from \cite{PachDZ17} that projects curves from $\RR^4$ down to $\RR^2$.
Instead, we rely on a property of parameterizations of Pfaffian curves. 
Algebraic curves do not satisfy this property. 

When trying to generalize expanding polynomial results such as \cite{RSdZ16,RSS16} to Pfaffian functions, we stumble upon similar issues.
In particular, these results rely on the property that a generic algebraic function is irreducible \cite{Stein89}. 
It is currently unknown whether such a property holds for Pfaffian functions. 

We also rely on the proximity technique of Solymosi and Zahl \cite{SolyZahl24}. 
This technique easily extends to the case of distinct distances between two curves.
However, it may not extend to other cases involving Pfaffian functions, due to the issues described above.

\parag{The paper structure.}   
In Section \ref{sec:pfaffian}, we define Pfaffian functions and Pfaffian sets in detail, and recall useful topological results about Pfaffian sets. 
In Section \ref{sec:proofoftheorems}, we prove Theorem \ref{th:twocurves}. 
The more algebraic portions of the proof are stated as lemmas whose proofs are postponed to Section \ref{sec:LemmaCij}.
In other words, Section \ref{sec:proofoftheorems} focuses on the combinatorial aspects of the proof, while Section \ref{sec:LemmaCij} focuses on the algebraic aspects.

\parag{Acknowledgments.} We thank Sida Li, Cosmin Pohoata, and Konrad Swanepoel for helpful discussions about related topics. AN would like to particularly thank Nicolai Vorobjov for answering questions on email. AN was supported by EPSRC Grant EP/V003542/1.

\section{Definitions and Preliminaries} 
\label{sec:pfaffian}

In this section, we study Pfaffian functions, Pfaffian sets, and other related objects.
For additional information about these topics, see the survey of Gabrielov and Vorobjov \cite{GV04}.

Let $\RR[x_1, \ldots, x_d]$ be the polynomial ring in coordinates $x_1, \ldots, x_d$, and we denote a point in $\RR^d$ as $\vecx = (x_1, \ldots, x_d)$.
Consider an open set $U \subseteq \RR^d$.
We recall that a function $f: U\to \RR$ is \emph{smooth} in $U$ if $f$ is infinitely differentiable in $U$.
The set of smooth functions in $U$ is $C^\infty(U)$.
We say that $\gamma\subset\RR^d$ is an \emph{arc} if there exists a continuous function from an infinite interval of $\RR$ to $\gamma$.
A \emph{curve} is the union of finitely many arcs.

\parag{Pfaffian chains and Pfaffian functions.}
Let $U\subseteq \RR^d$ be an open set.
A \emph{Pfaffian chain} over $U$ of \emph{order} $r$ and  \emph{chain-degree} $\alpha$  is a sequence of functions $\vec{q} = (q_1, \ldots, q_r) \in (C^{\infty}(U))^r$ that satisfies the following:
for every $1\le i \le r$ and $1 \le j \le d$, there exists a polynomial $P_{i, j} \in \RR[x_1, \ldots, x_d, y_1, \ldots, y_j]$ of degree at most $\alpha$ such that 
\begin{equation} \label{eq:pfaffian-differential-condition}
\frac{\partial q_j(\vecx)}{\partial x_i}   = P_{i,j}(\vecx, q_1(\vecx), \ldots, q_j(\vecx)).
\end{equation}

A function $f \in C^\infty(U)$ is called \emph{Pfaffian} on an open domain $U\subset \RR^d$ with respect to the Pfaffian chain $\vec{q} = (q_1, \ldots, q_r) \in (C^{\infty}(U))^r$ if there exists $Q_f \in \RR[x_1, \ldots, x_d, y_1, \ldots, y_r]$ of degree $\beta$ such that
\begin{equation} \label{eq:PfaffianFunc} 
f(\vecx) = Q_f(\vecx, q_1(\vecx), \ldots, q_r(\vecx)). 
\end{equation}
The \emph{order} of $f$ is the order $r$ of $\vec{q}$.
We say that $f$ has  \emph{format} $(\alpha,\beta,r)$.
We denote by $\Pf_d(\beta,\vec{q})$ the set of $d$-variate Pfaffian functions with respect to a Pfaffian chain $\vec{q}$ defined by a polynomial of degree at most $\beta$ from the ring $\RR[x_1, \ldots, x_d, y_1, \ldots, y_r]$. Thus we have $f\in \Pf_d(\beta, \vec{q})$.

The above is currently the standard and most common definition of Pfaffian functions, although other definitions exist in the literature. 
In particular, Khovanski{\u\i}'s original definition in \cite{Khovanskii91} is slightly more general than the above. Both definitions lead to nearly the same class of Pfaffian functions, albeit sometimes with different formats.

Examples of Pfaffian functions:
\begin{itemize}[itemsep=1pt,topsep=1pt]
\item A degree $D$ polynomial $f\in \RR[x_1, \ldots, x_d]$ is Pfaffian function of order 0 and chain-degree 0. 
That is, we may use an empty Pfaffian chain.
In this case, the format is $(0,D,0)$.
\item For $a\in \RR\setminus\{0\}$, $\vec{q} = (q_1)$ with $q_1(x)=e^{ax}$ is a Pfaffian chain of order 1 and chain-degree 1 given $\frac{\partial}{\partial x} q_1(x) = a q_1(x)$. Thus any $f \in \RR[x, e^{ax}]$ is a Pfaffian function with respect to $\vec{q}$.
\item Further generalizing the above, we note that $(e^{ax},e^{e^{ax}}, e^{e^{e^{ax}}},\ldots)$ is a Pfaffian chain. 
Thus, every $f\in \RR[x,e^{ax},e^{e^{ax}}, e^{e^{e^{ax}}},\ldots]$ is a Pfaffian function.
\item When $q_1(x) = \tan x$, the chain $\vec{q} = (q_1)$ is Pfaffian of order 1 and chain-degree 2, in the domain $U=\RR\setminus \{a+\pi/2\ :\ a\in\ZZ\}$. This is immediate from $\frac{\partial q_1(x)}{\partial x} = \mathrm{sec}^2 x = 1+q_1(x)^2$. Thus, every $f\in \RR[x, \tan x]$ is Pfaffian in $U$.
\item In $(0,\infty)$, for $q_1(x) = \frac{1}{x}$ and $q_2(x) = \ln x$, we have that $\vec{q} = (q_1, q_2)$ is a Pfaffian chain of order 2.
This is because $\frac{\partial q_1(x)}{\partial x} = -q_1(x)^2$ and $\frac{\partial q_2(x)}{\partial x} = q_1(x)$.
Thus, every $f\in \RR[x, \frac{1}{x}, \ln x]$ is Pfaffian in $(0,\infty)$.
\item In $\RR\setminus \{0\}$, for $q_1(x) = \frac{1}{x}$ and $q_2(x) = x^m$, $\vec{q} = (q_1, q_2)$ is a Pfaffian chain, for any real $m>0$.
This is immediate from the previous example and $\frac{\partial q_2(x)}{\partial x} = m q_1(x) q_2(x)$.
This implies that every $f\in\RR[x,\frac{1}{x},x^m]$ is Pfaffian in $\RR\setminus \{0\}$.
\item We can generalize the above examples to multivariate functions. 
For example, for $m_1,\ldots,m_d\in \RR$, the function $f(x_1,\ldots,x_d) = x_1^{m_1}\cdots x_d^{m_d}$ is Pfaffian due to the chain $(x_1^{-1},\ldots,x_d^{-1},x_1^{m_1}\cdots x_d^{m_d})$. The domain is $(0,\infty)^d$.
\end{itemize}

\parag{Pfaffian sets.}
For a $d$-variate function $f$ on domain $U \subseteq \RR^d$, we define the \emph{zero set} of $f$ on $U$ as 
\[ Z(f) = \{p\in U\ :\ f(p) = 0\}. \]

Consider an open $U\subset \RR^d$.
A subset $W \subseteq U$ is \emph{Pfaffian} if there exist Pfaffian functions $f_1,\ldots,f_t: U \to \RR$ such that 
\[ W = Z(f_1)\cap \cdots \cap Z(f_t). \]
 
A subset $W \subseteq U$ is \emph{semi-Pfaffian} if there exist Pfaffian functions $f_1,\ldots,f_t: U \to \RR$ and a Boolean function $\Phi: \{0, 1\}^t \to \{0, 1\}$ such that 
\begin{equation*}
W = \left\{ \vecx \in U\ :\ \Phi\left(f_1(\vecx) \ge 0, \ldots, f_t(\vecx) \ge 0\right) = 1 \right\}.
\end{equation*}
The \emph{order} of $W$ is the minimum length of a common Pfaffian chain $\vec{q}$ such that $f_1,\ldots,f_t$ are all Pfaffian with respect to $\vec{q}$.
The \emph{complexity} of $W$ is the coordinatewise maximum of all formats of $f_1,\ldots,f_t$.
These definitions hold for both Pfaffian and semi-Pfaffian sets.

We note that the notion of a set being semi-Pfaffian heavily relies on the domain $U$. 
A set may be semi-Pfaffian according to one domain and Pfaffian according to another. 
For example, $\{y=e^x\ : x>0\}$ is semi-Pfaffian in $\RR^2$ and Pfaffian in $(0,\infty)^2$.

For brevity, from here on, we omit reference to the domain $U$. 
We still implicitly assume that all occurring
Pfaffian functions are defined on the same domain.

The \emph{dimension} of a semi-Pfaffian set $W$ is defined to be the largest integer $k$ such that the set contains a homeomorphism of $[0,1]^k$.
The \emph{Zariski closure} of a set $W$, denoted $\overline{W}$, is the smallest Pfaffian set that contains $W$.
The Zariski closure is well-defined in the following sense: 
There exists a set $\overline{W}$ such that each Pfaffian set that contains $W$ also contains $\overline{W}$.
For brevity, we also refer to $\overline{W}$ as the \emph{closure} of $W$.

Consider a Pfaffian set $W$ that is defined by Pfaffian functions $f_1,\ldots,f_t$.
A point $p\in W$ is \emph{smooth} if all partial derivatives of all orders exist for each of $f_1,\ldots,f_t$ at $p$.
A point $p\in W$ is \emph{singular} if the the dimension of the tangent space of $W$ at $p$ does not equal to the dimension of $W$. 
This includes the case where the tangent space of $W$ at $p$ is not well-defined. 

We say that a Pfaffian set $W$ is \emph{reducible} if there exist Pfaffian sets $W_1$ and $W_2$ such that $W_1,W_2\neq W$ and $W_1 \cup W_2 = W$.
If $W$ is not reducible, then we say that it is \emph{irreducible}.

Next, consider a \emph{semi-}Pfaffian set $W\subset \RR^d$ that is defined by Pfaffian functions $f_1,\ldots,f_t$.
A point $p\in W$ is \emph{smooth} if $p$ is a smooth point of $\overline{W}$.
Similarly, $p\in W$ is \emph{singular} if it is a singular point of $\overline{W}$.
Finally, $W$ is irreducible if $\overline{W}$ is irreducible.

\parag{Stratification, components, and intersections.}
We now study some topological properties of semi-Pfaffian sets. 
While we state the dependence of some bounds in terms of the complexity of the sets, we only require bounds of the form $O_{\alpha, \beta,r,d}(1)$.

We require the following upper bound on the number of connected components of a Pfaffian set by Gabrielov and Vorobjov \cite[Corollary 3.3]{GV04}.

\begin{theorem} \label{th:ConnComp}
Let $\vec{q}$ be a Pfaffian chain of order $r$ and chain-degree $\alpha$. 
For any $d$-variate $f_1,\ldots, f_t \in \Pf_d(\beta, \vec{q})$ of format $(\alpha,\beta,r)$, the number of connected components of $Z(f_1)\cap\cdots \cap Z(f_t)$  is at most
\[ 2^{\binom{r}{2}+1} \beta (\alpha+2\beta-1)^{d-1}\left((2d-1)(\alpha+\beta)-2d+2\right)^r. \]
\end{theorem}

Consider a semi-Pfaffian $W\subset \RR^d$. 
A \emph{stratum} of $W$ is a smooth semi-Pfaffian subset of $W$. 
A $k$-dimensional stratum $S$ is \emph{basic} if there exist $d-k$ Pfaffian functions that vanish on $S$ and whose gradients are linearly independent at every point of $S$.
A \emph{basic weak stratification} of $W$ is a partition of $W$ into a disjoint union of basic strata. The following theorem establishes that every semi-Pfaffian set has a basic weak stratification

\begin{theorem} \label{th:Components}
Consider a semi-Pfaffian set $W \subseteq \RR^d$ of format $(\alpha,\beta,r)$. 
Then $W$ has a basic weak stratification. The number of strata, the number of connected components of each stratum, and the complexity of each stratum 
are $O_{\alpha, \beta,r,d}(1)$.
\end{theorem}

\begin{proof}
Immediate by combining the stratification result in \cite[Theorem 2]{GV95} by Gabrielov and Vorobjov and Theorem \ref{th:ConnComp} above.
\end{proof}

Given a semi-Pfaffian set $W\subset (-\mu, \mu)^d$, where $\mu > 0$, such that $W$ is defined by Pfaffian functions with maximum chain-degree $\alpha$, maximum degree $\beta$, and are all defined with respect to a common Pfaffian chain of length $r$, we will require a finer partition of $W$. 
In particular, we partition $W$ into \emph{cells}, defined as follows.
A $k$-dimensional cell with $k < d$ is the graph of a smooth map $f = (f_1,\ldots , f_{d-k})$ defined on an open set in $(-\mu, \mu)^k$. 
The functions $f_1,\ldots , f_{d-k}$ are Pfaffian of order $r$ and complexity $O_{\alpha,\beta,r,d}(1)$.
A $d$-dimensional cell is an open set.
All cells are \emph{basic} strata, as defined above. 
Theorem \ref{th:Components} can be used to obtain the following result.
For more details, see \cite[Lemma 3.1]{LNV24}.

\begin{theorem} \label{th:arcs}
For $\mu > 0$, let $W \subseteq (-\mu, \mu)^d$ be a semi-Pfaffian set of format $(\alpha,\beta,r)$.
Then there exists a partition of $W$ into cells of order $r$. The number of strata, the number of connected components of each stratum, and the complexity of each stratum is 
$O_{\alpha, \beta,r,d}(1)$.
\end{theorem}

Currently, no precise quantitative bounds are known for the number of irreducible components of a Pfaffian set. 
We rely on Theorem \ref{th:arcs} to get a coarse bound. 

\begin{corollary} \label{co:irreducible}
Let $\mu > 0$, let $W \subseteq (-\mu, \mu)^d$ be a semi-Pfaffian set of format $(\alpha,\beta,r)$.
Then the number of irreducible components of $W$ is $O_{\alpha, \beta,r,d}(1)$.
\end{corollary}
\begin{proof}
We partition $W$ into strata as described in Theorem \ref{th:arcs}. Consider an irreducible component $U$ of $W$ such that $(-\mu, \mu)^d \cap U$ is $k$-dimensional. By definition, there is a connected component $S$ of a $k$-dimensional stratum such that $S \cap (-\mu, \mu)^d \cap U$ is $k$-dimensional. We will now show that $S \subseteq U$, which in turn implies that the number of irreducible components of $W$ is at most the sum total number of connected components of the strata of $W$. Theorem \ref{th:arcs} states that this number is $O_{\alpha, \beta, r, d}(1)$.

Let $U$ be defined by a Pfaffian function $f$. By the definition of dimension, there exists a $k$-dimensional ball $B$ that is contained in $S$. That is, $f$ is identically zero on $B$. By the identity theorem (see for instance \cite[Section 1.2]{KP02}), $f$ vanishes on all of $S$. Thus $S \subseteq U$. This comples the proof.
\end{proof}

Finally, we need the following Bezout-type result for Pfaffian curves in $\RR^2$.

\begin{lemma} \label{le:pfBez}
Let $Q$ be a Pfaffian chain of order $r$ and chain-degree $\alpha$. 
Consider nonzero $f_1,f_2 \in \Pf_2(\beta,Q)$ of format $(\alpha,\beta,r)$. 
Then, either $Z(f_1)$ and $Z(f_2)$ have a common one-dimensional component or 
$|Z(f_1)\cap Z(f_2)| = O_{\alpha,\beta,r}(1)$.
\end{lemma}
\begin{proof}
We first consider the case where $Z(f_1)$ is zero-dimensional. 
In this case, Corollary \ref{co:irreducible} implies that 
\[ |Z(f_1)\cap Z(f_2)| \le |Z(f_1)| = O_{\alpha,\beta,r}(1), \]
as asserted.
A symmetric argument holds when $Z(f_2)$ is zero-dimensional.  
The only Pfaffian function that vanishes on the entire space is $0$.
Since $f_1$ and $f_2$ are nonzero, it remains to consider the case where both $Z(f_1)$ and $Z(f_2)$ are one-dimensional. 

We may assume that $Z(f_1)$ and $Z(f_2)$ do not share a common one-dimensional component, since otherwise we are done.
Then $Z(f_1) \cap Z(f_2)$ is a zero-dimensional Pfaffian set of order $r$.
In this case, Corollary \ref{co:irreducible} implies that $|Z(f_1)\cap Z(f_2)| = O_{\alpha,\beta,r}(1)$.
\end{proof}

\section{Proof of Theorem \ref{th:twocurves}}
\label{sec:proofoftheorems}

In this section, we prove Theorem \ref{th:twocurves}. 
More precisely, the current section contains the combinatorial infrastructure of the proof of Theorem \ref{th:twocurves}.
Several algebraic claims are stated as lemmas whose proofs are postponed to Section \ref{sec:LemmaCij}.

Let $\pts$ be a set of points in $\RR^2$ and let $\curves$ be a set of curves in $\RR^2$.
A pair $(p,\gamma)\in \pts\times \curves$ is an \emph{incidence} if $p$ lies on $\gamma$.
Let $I(\pts,\curves)$ be the number of incidences in $\pts\times \curves$.
For positive integers $k$ and $t$, we say that $\pts\times \curves$ has $k$ \emph{degrees of freedom} with \emph{multiplicity} $t$ if 
\begin{itemize}[itemsep=1pt,topsep=1pt]
\item For any $k$ points in $\curves$, at most $t$ curves of $\curves$ are incident to all $k$ points. 
\item Every pair of curves from $\curves$ intersect in at most $t$ points. 
\end{itemize}
\vspace{1mm}

We recall that a curve is \emph{simple} if it does not intersect itself. 
The following result is by Pach and Sharir \cite{PS98}.

\begin{theorem} \label{th:PachSharir}
Let $\pts$ be a set of $m$ points and let $\curves$ be a set of $n$ simple curves. 
If $\pts \times \curves$ has $k$ degrees of freedom with multiplicity $t$, then
\[ I(\pts,\curves) = O_{k,t} \left(m^{k/(2k-1)}n^{(2k-2)/(2k-1)}+m+n\right). \]
\end{theorem}

We now recall the statement of Theorem \ref{th:twocurves}.

\getkeytheorem{thtwocurves}

\begin{proof}
By possibly performing a rotation around the origin, we may assume that neither $\gamma_1$ nor $\gamma_2$ contains segments of vertical lines.
Such a rotation does not affect the number of distinct distances.

\parag{Pruning the point sets.} 
Let $\pts_1$ be a set of $m$ points on $\gamma_1$ and let $\pts_2$ be a set of $n$ points on $\gamma_2$. 
By possibly discarding points from $\pts_1$ and $\pts_2$, we may assume that these two sets are disjoint, that $|\pts_1|\ge m/2$, and that $|\pts_2|\ge n/2$.

We apply Theorem \ref{th:arcs} with a sufficiently large $\mu>0$, such that $\pts_1\subseteq (-\mu,\mu)^2$.
This implies that $\gamma_1$ can be partitioned into $O_{\alpha,\beta,r}(1)$ points and open arcs that are parametrized by Pfaffian functions of order $r$ and complexity $O_{\alpha,\beta,r}(1)$. 
By the pigeonhole principle, there exists an arc $\gamma_1'$ such that $|\gamma_1' \cap \pts_1| = \Theta_{\alpha,\beta,r}(m)$.
We note that $\gamma_1'$ is a graph, in the sense that every vertical line intersects $\gamma_1'$ at most once.
We perform a symmetric process for $\gamma_2$, to obtain $\gamma_2'$. 
As discussed in the proof of Corollary \ref{co:irreducible}, $\gamma_1'$ and $\gamma_2'$ are irreducible.

We repeat the argument in the preceding paragraph with the $x$- and $y$-axes switched. 
This leads to an arc $\gamma_1''\subseteq \gamma_1'$ that is a \emph{strictly monotone} arc and satisfies that $|\gamma_1'' \cap \pts_1| = \Theta_{\alpha,\beta,r}(m)$.
We perform a symmetric process for $\gamma_2$, to obtain a strictly monotone $\gamma_2'' \subseteq \gamma_2'$. 

If $\gamma_1''$ or $\gamma_2''$ is an arc of a  circle, then we remove the center of this circle from the other point set.
By assumption, if $\gamma_1''$ is a line segment, then $\gamma_2''$ is not a line segment parallel or orthogonal to $\gamma_1''$.
By Lemma \ref{le:pfBez}, $\gamma_2''$ intersects a line parallel or orthogonal to $\gamma_1''$ in $O_{\alpha,\beta,r}(1)$ points. 
Thus, by removing points from $\pts_2$, we may assume that each line parallel or orthogonal to $\gamma_1''$ contains at most one point of $\pts_2$ and that that $|\pts_2|=\Theta_{\alpha,\beta,r}(n)$.
Similarly, when $\gamma_1''$ is a circle, we may assume that any concentric circle contains at most one point of $\pts_2$. 
We perform a symmetric pruning of $\pts_1$ when $\gamma_2''$ is a line or a circle.

Abusing notation, we refer to the above $\gamma_1''$ as $\gamma_1$ and to $\gamma_1'' \cap \pts_1$ as $\pts_1$.
That is, $\gamma_1,\gamma_2$ are open arcs that are parameterized by Pfaffian functions of order $r$ and complexity $O_{\alpha,\beta,r}(1)$.
We also have that $|\pts_1|=\Theta_{\alpha,\beta,r}(m)$ and that $|\pts_2|=\Theta_{\alpha,\beta,r}(n)$.
With respect to the revised notation and pruned point sets, we write $m'=|\pts_1\cap \gamma_1|=\Theta_{\alpha,\beta,r}(m)$ and $n'=|\pts_2\cap \gamma_2|=\Theta_{\alpha,\beta,r}(n)$.

\parag{Distance energy and proximity.} 
We define the \emph{distance energy} of $\pts_1$ and $\pts_2$ as 
\[ E(\pts_1,\pts_2) = \left|\left\{(p,p',q,q')\in \pts_1^2\times \pts_2^2\ :\ |pq|=|p'q'| \right\}\right|.\]
Let $D$ be the set of distances between $\pts_1$ and $\pts_2$.
For $d\in D$, we set 
\[ r_d = \left|\{(p,q)\in \pts_1\times \pts_2\ :\ |pq| = d\}\right|. \] 
In other words, $r_d$ is the number of representations of $d$ as a distance in $\pts_1\times\pts_2$. 

Since every pair in $\pts_1\times\pts_2$ spans a distance, we get that $\sum_{d\in D} r_d = m'n'$.
The Cauchy-Schwarz inequality implies that
\begin{equation} \label{eq:EnergyLower} E(\pts_1,\pts_2) = \sum_{d\in D} r_d^2
  \geq \frac{1}{|D|} \left( \sum_{d\in D} r_d   \right)^2 = \frac{(m'n')^2}{|D|}. 
\end{equation}

We write $\pts_1 = \{p_1,\ldots,p_{m'}\}$ and $\pts_2 = \{q_1,\ldots,q_{n'}\}$, both ordered in increasing order of $x$-coordinates.
Since $\gamma_1$ and $\gamma_2$ are strictly monotone, the sequence of $y$-coordinates in each set is either strictly increasing or strictly decreasing.

For a parameter $0<c<1$ that is set below, we define
\emph{distance energy with proximity} $c$ as 
\[ E_c(\pts_1,\pts_2) = \left|\left\{(p_i,p_j,q_{i'},q_{j'})\in \pts_1^2\times \pts_2^2\ :\ |pq|=|p'q'|,\ |i-j|\le cm',\ |i'-j'|\le cn' \right\}\right|.\]

By combining \eqref{eq:EnergyLower} with a simple counting argument from \cite[Section 2]{SolyZahl24}, we have that
\begin{equation} \label{eq:EnergyLowerc} E_c(\pts_1,\pts_2) =\Omega(c \cdot E(\pts_1,\pts_2)) = \Omega\left(c\cdot \frac{(m'n')^2}{|D|}\right). 
\end{equation}

With \eqref{eq:EnergyLowerc} in mind, it remains to derive an upper bound for $E_c(\pts_1,\pts_2)$.

\parag{Point-curve incidences.}
For $1\le i,j\le m'$, we define
\begin{equation} \label{eq:Cij} 
C_{i,j} = \{(q,q')\in \gamma_2^2 \subset\RR^4\ :\ |p_i q| = |p_jq'| \}. 
\end{equation}
Denoting the coordinates of $\RR^4$ as $x_1,x_2,x_3,x_4$, we have that
$C_{i,j}$ is a connected open subset of the set of points $(q,q')\in \RR^4$ that satisfy
\begin{align*}
 (x_1-p_{i,x})^2 + (x_2-p_{i,y})^2 &= (x_3-p_{j,x})^2 + (x_4-p_{j,y})^2, \nonumber \\[2mm]
  f_2(x_1,x_2) &= f_2(x_3,x_4)=0. 
  \end{align*}

The above equations imply that $C_{i,j}$ is Pfaffian of order $r$ and complexity $O_{\alpha,\beta,r}(1)$.
We refer to $C_{i,j}$ as a curve, since below we show that it is one-dimensional. 
We set 
\[ \curves = \{ C_{i,j}\ :\ 1\leq i,j \leq m' \text{ and } |i-j|\le m'c\}.\]
In theory, $\curves$ may be a multi-set, containing identical elements.
We note that $|\curves| =\Theta(c\cdot (m')^2)$. 

We set $\pts_c = \{(q_i,q_j)\in \pts_2^2\ :\ |i-j|\le cn'\}$.
By the above, a point $(q_{i'},q_{j'})\in \pts_c$ is incident to a curve $C_{i,j}$ if and only if $|p_i q_{i'}| = |p_jq_{j'}|$.
In other words: $E_c(\pts_1,\pts_2)= I(\pts_c,\curves)$. 
Thus, it remains to obtain an upper bound for $I(\pts_c,\curves)$.

To obtain a good upper bound on the number of incidences, we need to address the possibility of $\curves$ potentially containing identical curves.
As a first step, the following lemma collects properties of the intersections of the curves $C_{i,j}$. 
We prove it in Section \ref{sec:LemmaCij}.

\begin{lemma}[store=leCijInt]\label{le:CijInt}
\begin{enumerate}[label=(\alph*)]
\item \label{cij:curve} Each $C_{i,j} \in \curves$ is of dimension at most one.
\item \label{cij:no-three-infinite} There exists a subset $\Gamma_0\subset \Gamma$ such that  $|\Gamma_0|=O_{\alpha,\beta, r}(m')$ and no three curves in $\Gamma\backslash \Gamma_0$ have an infinite intersection.
\item\label{cij:finite-small} If two curves $C_{i,j},C_{k,\ell}\in \Gamma$ have a finite intersection then $|C_{i,j}\cap C_{k,l}| = O_{\alpha, \beta, r}(1)$. 
\end{enumerate}
\end{lemma}

Let $\curves_0$ be as defined in Lemma \ref{le:CijInt}-\ref{cij:no-three-infinite}.  
By Corollary \ref{co:irreducible} and Theorem \ref{th:ConnComp}, there exists $t'=O_{\alpha,\beta,r}(1)$, such that each curve in $\curves\setminus\curves_0$ has at most $t'$ irreducible connected components.
We apply Corollary \ref{co:irreducible} with a sufficiently large $\mu$, so that $(-\mu, \mu)^4$ contains all of the components. Some of these components may extend to infinity beyond $(-\mu, \mu)^4$.
We note that each such component is either a single point or an arc. 
We create sets $\curves_1,\ldots,\curves_{t'}$ that consist of the irreducible connected components of the elements of $\curves\setminus\curves_0$.
In particular, we break each curve of $\curves\setminus\curves_0$ into irreducible components, and place each component in a different $\curves_i$.

We consider a fixed set $\curves_i$.
By Lemma \ref{le:CijInt}, no three elements of $\curves_i$ have an infinite intersection. 
For two arcs to have an infinite intersection, they must belong to the same irreducible curve. 
We claim that we can partition $\curves_i$ into three subsets, such that every two elements from the same subset have a finite intersection. 
Indeed, consider a connected component $\gamma$ that includes overlapping arcs from $\curves_i$. 
We travel along $\gamma$ starting at an arbitrary point. 
Each time we encounter the beginning of a new arc $a\in \curves_i$, we are inside at most one other arc $a'\in \curves_i$. 
We insert $a$ to a different subset of $\curves_i$ than that of $a'$.
The third subset is required only when $\gamma$ is a closed curve, and only for the case where the last arc we visit overlaps with an arc that contained our starting point. 
We repeat the above process for every relevant curve $\gamma$. 

By the preceding paragraph, we partitioned the curves of $\curves\setminus\curves_0$ into $O_{\alpha,\beta,r}(c\cdot(m')^2)$ points and arcs.
These points and arcs are grouped into $t\le 3t' = O_{\alpha,\beta,r}(1)$ sets $\curves_1,\ldots,\curves_t$, such that no two elements from the same set have an infinite intersection. 
We note that 
\[ I(\pts_c,\curves) = \sum_{i=0}^{t}I(\pts_c,\curves_i). \]
 
\parag{Simplifying the incidence problems.} 
For $0 < i \le t$, we wish to use Theorem \ref{th:PachSharir} to obtain a bound for $I(\pts_c,\curves_i)$.
To be able to do that, we now reduce the problem to an incidence problem in $\RR^2$ with two degrees of freedom.

We reverse the roles of $\gamma_1$ and $\gamma_2$ in the above analysis, as follows.
Symmetrically to \eqref{eq:Cij}, we set  
\[ \widetilde{C}_{i,j} = \{(p_{i'},p_{j'})\in \gamma_1^2\ :\ |q_i p_{i'}| = |q_jp_{j'}|\} \subset \RR^4. \]
In other words, $\widetilde{C}_{i,j}$ is a connected open subset of the points $(x_1,x_2, x_3,x_4)$ satisfying
 \begin{align*}
 (q_{i,x}-x_1)^2 + (q_{i,y}-x_2)^2 &= (q_{j,x}-x_3)^2 + (q_{j,y}-x_4)^2, \\[2mm]
  f_1(x_1,x_2)=0,&~~~f_1(x_3,x_4)=0.
 \end{align*}
 
We set $\widetilde{\curves} = \{ \widetilde{C}_{i,j}\ :\ 1\leq i,j \leq n' \text{ and } |i-j|\le c n'\}$.
By the statement symmetric to Lemma \ref{le:CijInt}-\ref{cij:no-three-infinite},
 there is a subset $\widetilde{\curves}_0\subset \widetilde{\curves}$ of size $O_{\alpha,\beta,r}(n')$, such that no three curves of $\widetilde{\curves}\setminus \widetilde{\curves}_0$ have an infinite intersection.
Let $\pts_0$ be the set of points $(q_i,q_j)\in \pts_2^2$ 
that correspond to the curves of $\widetilde{\curves}_0$.
 
By repeating the above argument for partitioning $\curves\setminus \curves_0$ into $t$ sets of points and arcs, we obtain a partition of $\pts\backslash \pts_0$ into subsets $\pts_1,\ldots, \pts_s$, so that two points from the same $\pts_j$ correspond to curves with a finite intersection.
Here, $s=O_{\alpha,\beta,r}(1)$.

By the analog of Lemma \ref{le:CijInt}-\ref{cij:finite-small}, for any two points from the same $\pts_j$, the number of curves from $\curves$ that are incident to both is $O_{\alpha,\beta,r}(1)$.  
We conclude that, for all $1\le i \le t$ and $1\le j\le s$,  the configuration $\pts_i\times\curves_j$ has two degrees of freedom with multiplicity $O_{\alpha,\beta,r}(1)$.

To recap, when ignoring $\pts_0$ and $\curves_0$, we have $st$ incidence problems, each with two degree of freedom. 
However, these incidence problems are in $\RR^4$. 
Currently, we do not have a reasonable incidence bound for Pfaffian curves in $\RR^4$.
Also, as mentioned earlier, projections of Pfaffian sets are not well-behaved.
Fortunately, in this specific case, we can perform a different operation that is equivalent to a projection. 

\begin{lemma}\label{le:mainincidences}
For any $1\le i \le t$ and $1\le j\le s$, we have that
\[ I(\pts_i,\curves_j) = O_{\alpha,\beta,r}\left(c^{4/3}(m')^{4/3}(n')^{4/3} + c\cdot(m')^2 + c\cdot (n')^2\right).\]
\end{lemma}
\begin{proof}
Recall that $f_2(x,y)$ is the Pfaffian function whose zero set is the arc $\gamma_2$, in a corresponding open set. 
Let $y=g_2(x)$ be the Pfaffian parameterization of $\gamma_2$ in this open set. 

Consider $p,p'\in \pts_1$ and $q,q'\in P_2$. 
Asking whether the point $(q,q')\in \RR^4$ is incident to the curve $C_{p,p'}$ leads to the system 
\begin{align*}
(p_x-q_x )^2+(p_y-q_y )^2 &= (p_x'-q_x')^2+(p_y'-q_y')^2 \\[2mm]
f_2(q_x,q_y)&=f_2(q'_x,q'_y)=0.
\end{align*}

Using the parameterization of $\gamma_2$, the above system becomes
\[ (p_x-q_x)^2+(p_y-g_2(q_x ))^2=(p_x'-q_x' )^2+(p_y'-g_2(q_x'))^2. \]

By taking $q_x$ and $q_x'$ to be coordinates of a plane $\RR^2$, we get the bivariate Pfaffian equation
\[ (p_x-x)^2+(p_y-g_2(x))^2=(p_x'-y)^2+(p_y'-g_2(y))^2. \]
This defines a Pfaffian curve in $\RR^2$, which is the projection of $C_{p,p'}\subset\RR^4$ on the $(q_x,q_x')$-plane. 
We denote this projection as $\Pi(q_x,q_y,q'_x,q'_y) = (q_x,q_x')$.

We recall that $\gamma_2$ is a graph.
This implies that the projection $\Pi$ from $\gamma_2^2$ is one-to-one. 
That is, a point-curve pair in $\RR^4$ form an incidence if and only if their projections form an incidence in $\RR^2$.
Similarly, two curves intersect at a point $u\in \RR^2$ if and only if the two corresponding curves in $\RR^4$ intersect at $\Pi^{-1}(u)$.  
Thus, the incidence structure is maintained and the planar configuration has two degrees of freedom with multiplicity $O_{\alpha,\beta,r}(1)$. 
 
To complete the proof of the lemma, we apply Theorem \ref{th:PachSharir} on the above configuration in $\RR^2$. 
\end{proof}

\parag{Wrapping up.}
It remains to study incidences with $\pts_0$ and with $\curves_0$.
We fix a curve  $C_{i,j}\in\Gamma_0$ and a point $q\in \pts_2$.
For a point $(q,q')\in \pts_2^2$ to be incident to $C_{i,j}$, the point $q'\in\gamma_2$ must be on a circle of radius $|p_iq|$ centered at $p_j$.
We recall that, if $\gamma_2$ is an arc of a circle, then the center of the circle is not in $\pts_1$.
Combining this with Lemma \ref{le:pfBez} implies that the number of points $q'\in\pts_2$ that satisfy the above is $O_{\alpha,\beta,r}(1)$. 
Recalling that $|\curves_0|=O_{\alpha,\beta,r}(m')$ and summing up the above over all $C_{i,j}\in\Gamma_0$ leads to $I(\pts,\curves_0)=O_{\alpha,\beta,r}(m'n')$. 
 
We obtain an upper bound for $I(\pts_0,\curves)$ symmetrically. 
That is, we fix $(q,q')\in \pts_0$ and $1\le i \le t$. 
To have that $(q,q')\in C_{i,j}$, we require $p_j$ to be on the circle of radius $|p_iq|$ centered at $q'$.
This leads to $O_{\alpha,\beta,r}(1)$ possible values of $j$. 
Since $|\pts_0|= O_{\alpha,\beta,r}(n')$, summing up the above over every $(q,q')\in \pts_0$ leads to $I(\pts_0,\curves)=O_{\alpha,\beta,r}(m'n')$. 

By combining the above with Lemma \ref{le:mainincidences}, we obtain that 
 \begin{align*}
|I(P_c,\Gamma)| &\leq |I(P_0, \Gamma)| + |I(P,\Gamma_0)| + \sum_{i,j\ge 1} |I(P_i,\curves_j)| \\
   &= O_{\alpha,\beta,r}(m'n') +\sum_{i,j\ge 1} O_{\alpha,\beta,r}(c^{4/3} (m')^{4/3}(n')^{4/3} + c\cdot(m')^2 + c\cdot('n)^2) \\
   &=O_{\alpha,\beta,r}( c^{4/3}(m')^{4/3}(n')^{4/3}+ c\cdot(m')^2+ c\cdot(n')^2+m'n').
 \end{align*}

We recall that $E_c(\pts_1,\pts_2)= I(\pts_c,\curves)$.
Combining this with \eqref{eq:EnergyLowerc} leads to
\[ c\cdot \frac{(m'n')^2}{|D|} \leq E_c(\pts_1,\pts_2) = |I(P_c,\Gamma)| = O_{\alpha,\beta,r}\left(c^{4/3}(m')^{4/3}(n')^{4/3} + c\cdot(m')^2 + c\cdot(n')^2+m'n'\right). \]

In the above inequality, we wish to avoid the case where the term $m'n'$ is larger than the expression to the left.
This happens when $c=\Omega_{\alpha,\beta,r}(|D|/m'n')$.
Since we wish to minimize $c$, we set $c=\Theta_{\alpha,\beta,r}(|D|/m'n')$. 
This turns the above inequality to 
\[ m'n' = O_{\alpha,\beta,r}\left(|D|^{4/3} + \frac{|D|m'}{n'} + \frac{|D|n'}{m'}\right). \]

Each of the three terms on the right-hand side may dominate the bound. 
By separately considering each case and recalling that $m'=\Theta_{\alpha,\beta,r}(m)$ and $n'=\Theta_{\alpha,\beta,r}(n)$, we obtain that $|D|=\Omega_{\alpha,\beta,r}(\min\{m^{3/4}n^{3/4},n^2,m^2\})$.
Since bringing back the points that were removed from $\pts_1$ and $\pts_2$ can only increase the number of distances, the proof is complete.
\end{proof}  

\section{Proof of Lemma \ref{le:CijInt}} \label{sec:LemmaCij}

In this section, we prove Lemma \ref{le:CijInt}.
This proof contains some of the more algebraic parts of the proof of Theorem \ref{th:twocurves}.
We first recall the statement of the lemma and the definitions relevant to it.
For $1\le i,j\le m'$, we consider the set \begin{equation*}  
C_{i,j} = \{(q,q')\in \gamma_2^2 \subset\RR^4\ :\ |p_i q| = |p_jq'| \}. 
\end{equation*}
We also have that $\curves = \{ C_{i,j}\ :\ 1\leq i,j \leq m' \}$.
\vspace{2mm}

\getkeytheorem{leCijInt}

We first prove parts (a) and (c).
The proof of part (b) is significantly longer and more involved, so it has its own subsection below.

\begin{proof}[Proof of Lemma \ref{le:CijInt}-\ref{cij:curve}.]
A $C_{i,j}\in \curves$ is the set of points $(q,q')\in \RR^4$ that satisfy  
\begin{align*}
(p_{i,x}-q_x )^2+(p_{i,y}-q_y )^2 &= (p_{j,x}-q_x')^2+(p_{j,y}-q_y')^2 \\[2mm]
f_2(q_x,q_y)&=f_2(q'_x,q'_y)=0.
\end{align*}

As in the proof of Lemma \ref{le:mainincidences}, let $\gamma_{i,j}$ be the curve in $\RR^2$ that is defined by
\[ (p_{i,x}-x)^2+(p_{i,y}-g_2(x))^2=(p_{j,x}-y)^2+(p_{j,y}-g_2(y))^2. \]
As discussed in the proof of Lemma \ref{le:mainincidences}, the projection $\Pi: \gamma_2^2 \to \RR^2$ is a bijection between $C_{i,j}$ and $\gamma_{i,j}$.

The only bivariate Pfaffian function with a two-dimensional zero set in $\RR^2$ is 0, so the dimension of $\gamma_{i,j}$ is at most one. 
In other words, $\gamma_{i,j}$ does not have a subset homeomorphic to $[0,1]^2$.
By the bijection $\Pi$, the set $C_{i,j}$ also does not contain such a subset and is thus of dimension at most one. 
\end{proof}

\begin{proof}[Proof of Lemma \ref{le:CijInt}-\ref{cij:finite-small}]
We write $p_i=(p_{i,x},p_{i,y})$.
Then, a point $(x_1,x_2,x_3,x_4)\in C_{i,j}\cap C_{k,\ell}$ satisfies 
 \begin{align}
  (x_1-p_{i,x})^2 + (x_2-p_{i,y})^2 &= (x_3-p_{j,x})^2 + (x_4-p_{j,y})^2, \nonumber \\[2mm]
   (x_1-p_{k,x})^2 + (x_2-p_{k,y})^2 &= (x_3-p_{\ell,x})^2 + (x_4-p_{\ell,y})^2, \nonumber \\[2mm]
  f_2(x_1,x_2) &= f_2(x_3,x_4)=0. \label{eq:CijTwice}
 \end{align}
Each of the above equations is Pfaffian of format at most $(\alpha,\beta,r)$. 
Theorem \ref{th:ConnComp} implies that $C_{i,j}\cap C_{k,\ell}$ has $O_{\alpha,\beta,r}(1)$ connected components.
If $C_{i,j}\cap C_{k,\ell}$ is finite, then it consists of $O_{\alpha,\beta,r}(1)$ points.
\end{proof}

\subsection{Symmetries of Pfaffian sets}

We now study preliminaries that are required for the proof of Part (b) of Lemma \ref{le:CijInt}.
This is a variant of an analysis of Pach and de Zeeuw \cite{PachDZ17}, generalized to Pfaffian curves.

The \emph{isometries} of $\RR^2$ are the rotations, translations, and glide reflections (a reflection
combined with a translation).
A transformation $T$ is a \emph{symmetry} of a curve $\gamma$ if $T$ is an isometry of $\RR^2$ that satisfies $T(\gamma) = \gamma$.
We now list facts about isometries in $\RR^2$. 
For more information, see Conrad \cite{Conrad}.
Let $I_2$ be the $2\times 2$ identity matrix.
Let $O_\RR(2)$ be the group of orthogonal $2\times 2$ matrices. 

\begin{enumerate}[label=(\textbf{Iso-\arabic*}),ref=(\textbf{Iso-\arabic*})]
    \item \label{isofact:def} The isometries of $\RR^2$ are the transformations 
    \[ h_{A, \vecw}(\vecv) = A\vecv + \vecw,\]
    where $A \in O_\RR(2)$ and $\vecw \in \RR^2$.
    \item \label{isofact:classification} A complete classification of isometries of $\RR^2$:
    \begin{enumerate}
        \item The \emph{identity} is obtained when $A = I_2$ and $\vecw = 0$. 
        \item A nonzero \emph{translation} is obtained when $A = I_2$ and $\vecw \neq 0$. 
        \item A nonzero \emph{rotation} is obtained when $\det(A)=1$ and $A\neq I_2$. These are the matrices $A = \begin{pmatrix}
            \cos(\theta) & -\sin(\theta) \\ \sin(\theta) & \cos(\theta)
        \end{pmatrix}$ with $\theta \neq 0$. The center of the rotation is $p = (I_2 - A)^{-1}\vecw$ and the angle is $\theta$. We may also write $h(\vecv)=A(\vecv - p) + p$.
        \item A \emph{reflection} is obtained when $\det(A)=-1$ and $A\vecw = -\vecw$. In particular, it is the reflection across line $\frac{\vecw}{2} + \mathrm{ker}(A - I_2)$. We may also write $h(\vecv) = A\left(\vecv - \frac{\vecw}{2}\right) + \frac{\vecw}{2}$.
        \item A \emph{glide reflection} is obtained when $\mathrm{det}(A) = -1$ and $A\vecw \neq -\vecw$. 
    \end{enumerate}
    \item \label{isofact:group} The set of symmetries of a curve forms a group under composition.
\end{enumerate}

We require an upper bound on the number of symmetries a curve may have.

\begin{lemma}\label{lem:syms}
Let $Q$ be a Pfaffian chain with chain-degree $\alpha$ and order $r$. 
Consider $f \in \Pf_2(\beta, Q)$ such that $\gamma= Z(f)\subset \RR^2$ is an irreducible Pfaffian curve that is neither a line nor a circle. Then $\gamma$ has $O_{\alpha,\beta,r}(1)$ symmetries.
\end{lemma}
\begin{proof}
Assume for contradiction that $\gamma$ has {\bf a translation symmetry} $T$.
We write $T(\vecv) = \vecv+\vecw$, with $\vecw\in \RR^2\setminus \{0\}$. 
For a point $p\in \gamma$, we have that $f(p) = 0$, which in turn implies that $f(p+\vecw)=0$. 
Inductively, this means that $f(p + n\vecw) = 0$ for all $n \in \ZZ$. 
That is, $\gamma$ has an infinite intersection with the line $\{p + t\vecw : t\in \RR \}$.
By Theorem \ref{le:pfBez}, $\gamma$ is this line, contradicting the assumption.  
  
Next, we assume that $\gamma$ has two {\bf rotation symmetries}.
Following \ref{isofact:def}, we denote these rotations as $h_1=h_{A_1, \vecw_1}(\vecv)$ and $h_2=h_{A_2, \vecw_2}(\vecv)$.
We denote the centers of the rotations as $p_1$ and $p_2$, respectively.
By \ref{isofact:classification}, we have that $h_{i}(\vecv) = A_i(\vecv-p_i) + p_i$. 
We note that $h_{i}(\vecv)^{-1} = A_i^{-1}(\vecv-p_i) + p_i$.
We recall that $A_1,A_2,I_2$ are in the group $O_\RR(2)$, which is Abelian. 
Combining the above leads to
  \begin{align*}
 &(h_2^{-1} \circ h_1^{-1} \circ h_2 \circ h_1)(\vecv) \nonumber \\
 &\qquad \qquad = (h_2^{-1} \circ h_1^{-1} \circ h_2)(A_1(\vecv - p_1) + p_1) \nonumber \\
 &\qquad \qquad = (h_2^{-1} \circ h_1^{-1})(A_2A_1(\vecv - p_1) + A_2 p_1 - (A_2 - I_2)p_2) \nonumber \\
 &\qquad \qquad = (h_2^{-1})(A_2(\vecv - p_1) + (A_1^{-1}A_2 - A_1^{-1} + I_2)p_1 - (A_1^{-1}A_2 - A_1^{-1})p_2) \nonumber \\
 &\qquad \qquad = \vecv + (A_1^{-1} - A_2^{-1}A_1^{-1} + A_2^{-1} - I_2) p_1 - (A_1^{-1} - A_2^{-1}A_1^{-1} + A_2^{-1} - I_2) p_2 \nonumber \\
 &\qquad \qquad = \vecv + (I_2 - A_2^{-1})(I_2 - A_1^{-1})(p_2 -p_1). \nonumber
 \end{align*}

By \ref{isofact:group}, the above transformation is a symmetry of $\gamma$.
When $p_1\neq p_2$, this symmetry is a translation. 
By the first paragraph of the current proof, $\gamma$ cannot have a symmetry that is a translation.
Thus, all rotational symmetries of $\gamma$ have the same center $p$.

Fix a point $q\in \gamma\setminus\{p\}$.  The images of $q$ under the rotational symmetries of $\gamma$ are distinct points on a circle $C$ with center $p$.
By Lemma \ref{le:pfBez}, either $|C \cap \gamma|=O_{\alpha,\beta,r}(1)$ or $C=\gamma$.
By the assumption that $\gamma$ is not a circle, we have that $|C \cap \gamma|=O_{\alpha,\beta,r}(1)$.
This in turn implies that $\gamma$ has $O_{\alpha,\beta,r}(1)$ rotational symmetries. 

Assume for contradiction that $\gamma$ has two {\bf reflection symmetries} with parallel axes of symmetry. 
Combining two such reflections leads to a translation. 
By \ref{isofact:group}, this translation is also a symmetry of $\gamma$, which contradicts the above.
Thus, each two reflection symmetries of $\gamma$ have non-parallel axes.

Consider two reflection symmetries of $\gamma$.
By the preceding paragraph, the two respective axes of reflection intersect in a point $p$. 
Following \ref{isofact:classification}, we may write
\[ h_1(\vecv) = A_1\left(\vecv - \frac{\vecw_1}{2}\right) + \frac{\vecw_1}{2} = A_1\vecv + u_1.  \]
Similarly, we write $h_2(\vecv) = A_2\vecv + u_2$.
Since $p$ is fixed by $h_1$, we have that $A_1 p + u_1 = p$.
Since $p$ is also fixed by $h_2$, applying $h_2$ to both sides above leads to
\begin{equation*} 
A_2(A_1 p + u_1) + u_2 = p \quad \Rightarrow \quad (I_2-A_2A_1)p = A_2 u_1 + u_2.  
\end{equation*}

Combining the above implies that
\[ (h_2 \circ h_1)(\vecv) = A_2A_1\vecv + A_2u_1 + u_2 = A_2A_1\vecv + (I_2 - A_2A_1)p. \]
By \ref{isofact:classification}, $\det(A_1)=\det(A_2)=-1$, so $\det(A_2A_1)=1$.
That is, the symmetry $(h_2 \circ h_1)(\vecv)$ is a rotation with center $p$.
By the above argument for rotation symmetries, the lines of all reflection symmetries of $\gamma$ intersect at $p$.
We fix one such reflection symmetry $h$ and note that combining it with each other reflection symmetry of $\gamma$ leads to a distinct rotation around $p$.
The above analysis of rotation symmetries implies that $\gamma$ has $O_{\alpha,\beta,r}(1)$ reflection symmetries. 

By \ref{isofact:classification}, it remains to consider the case of {\bf a glide reflection symmetry} $h$ of $\gamma$.
We write $h(\vecv) = A\vecv + \vecw$, where  $A\vecw \neq \vecw$ and $A^2=I_2$.
By \ref{isofact:group}, $(h \circ h)(\vecv) = \vecv + (I_2 + A)\vecw$ is a symmetry of $\gamma$.
However, this is a translation, which is impossible. 
We thus conclude that $\gamma$ has no glide reflection symmetries.
\end{proof}

In Section \ref{sec:Lemmab}, we sometimes refer to \emph{symmetries of a Pfaffian arc} $\gamma$. 
This refers to symmetries of the Pfaffian curve $\overline{\gamma}$, as described above.

\subsection{Proof of Lemma \ref{le:CijInt}-\ref{cij:no-three-infinite}} \label{sec:Lemmab}

We now introduce four lemmas about the curves $C_{i,j}$.
Afterwards, we show how combining these four lemmas leads to Lemma \ref{le:CijInt}-\ref{cij:no-three-infinite}.

\begin{lemma}\label{le:samedist}
 If $|p_i p_k| = |p_j p_\ell|$ and $C_{i,j}\cap C_{k,\ell}$ is infinite,  then $\gamma_2$ has a symmetry that maps $p_i$ to $p_j$ and $p_k$ to $p_\ell$.
\end{lemma}
\begin{proof}
 A point $(q_{1,x},q_{1,y},q_{2,x},q_{2,y}) = (q_1,q_2)\in C_{i,j}\cap C_{k,\ell}$ 
 corresponds to $q_1, q_2\in \gamma_2$ 
 that satisfy
\[ |p_i q_1| = |p_j q_2| \quad \text{ and }\quad |p_k q_1| = |p_\ell q_2|. \] 
This implies that 
 \begin{align*}
  \begin{split}
   \{ (d_1,d_2)\ :\ \text{exists } (q_1, q_2)\in C_{i,j}\cap C_{k,\ell} \text{ such that } d_1 = |p_i q_1|\text{ and } d_2 = |p_k q_1|\}\\[2mm]
   = \{ (d_1,d_2)\ :\ \text{exists } (q_1, q_2)\in C_{i,j}\cap C_{k,\ell} \text{ such that } d_1 = |p_j q_2|\text{ and } d_2 = |p_l q_2|\}.
  \end{split}
 \end{align*}
We denote this set of distance pairs as $D$.

For a fixed distance pair $(d_1,d_2)\in D$, we consider the number of pairs $(q_1,q_2)\in C_{i,j}\cap C_{k,\ell}$ that satisfy $d_1 = |p_i q_1|$ and $d_2 = |p_k q_1|$.
We note that $q_1$ is on a circle of radius $d_1$  and center $p_i$.
By an assumption from Section \ref{sec:proofoftheorems}, the curve $\gamma_2$ is not a circle centered at $p_i$.
Thus, Lemma \ref{le:pfBez} implies that there are $O_{\alpha,\beta,r}(1)$ options for $q_1$.
A symmetric argument holds for $q_2$, so each pair $(d_1,d_2)\in D$ is obtained from $O_{\alpha,\beta,r}(1)$ points from $C_{i,j}\cap C_{k,\ell}$.
Since $C_{i,j} \cap C_{k,\ell}$ is infinite,
we conclude that $D$ is also infinite.
  
We define 
\[ E = \{ q\in \RR^2\ :\ \left (|p_i q|, |p_k q|\right)\in D\}. \]
We also set $E_1 = E\cap \gamma_2$. 
Since $D$ is infinite, so is $E_1$.
Since $|p_ip_k|=|p_jp_\ell|$, there exists a rigid motion $T$ that maps $p_i$ to $p_j$ and $p_k$ to $p_\ell$.
Since $D$ is infinite, the motion $T$ takes infinitely many points of $E$ to $\gamma_2$.
Let $E_1^*\subset E$ be the set of these points of $E$. Let $R_{i,j}$ be the reflection about the line through $p_i$ and $p_k$ and set $E_2^*=R_{i,j}(E_1^*)$.
We note that at least one of $E_1\cap E_1^*$ and $E_1\cap E_2^*$ is infinite.

We first consider the case where $E_1\cap E_1^*$ is infinite. 
We write $F_1 = E_1\cap E_1^*$ and note that $T(F_1)\subset \gamma_2$.
Since $F_1 \subset T(\gamma_2)$, we obtain that $\gamma_2\cap T(\gamma_2)$ is infinite. 
Rigid motions preserve Pfaffian curves and their format, so $T(\gamma_2)$ is Pfaffian. 
Lemma \ref{le:pfBez} implies that $\gamma_2$ and $T(\gamma_2)$ have a common component. 
Since $\gamma_2$ is irreducible, we get that $\gamma_2=T(\gamma_2)$.
That is, $T$ is a symmetry of $\gamma_2$ that maps $p_i$ to $p_j$ and $p_k$ to $p_\ell$.

It remains to consider the case where $E_1\cap E_2^*$ is infinite.
In this case, we repeat the analysis of the preceding paragraph, but with $T\circ R_{i,j}$ instead of $T$.
\end{proof}

We now study the case when  $|p_ip_k| \neq |p_j p_\ell|$.
More precisely, we consider triples of curves that satisfy this property.

\begin{lemma}\label{lem:diffdist}
Consider $p_i,p_j,p_k,p_\ell,p_s,p_t\in \pts_1$ such that 
 $|p_i p_k| \neq |p_j p_\ell|$, 
 $|p_i p_s| \neq |p_j p_t|$, and 
 $|p_k p_s| \neq |p_\ell p_t|$.
If $\gamma_2$ is non-algebraic or $\gamma_2$ is algebraic of degree at least 3, then 
\[|C_{i,j}\cap C_{k,\ell} \cap C_{s,t}| =O_{\alpha,\beta,r}(1). \]
\end{lemma}
\begin{proof}
A point in $(q_1,q_2)\in C_{i,j}\cap C_{k,l} \cap C_{q,r}$ corresponds 
 to two points $q_1,q_2\in \gamma_2$
 such that the distances between $q_1$ and $p_i,p_k,p_s$ are equal to the distances between $q_2$ and $p_j, p_\ell,p_t$, respectively.
By potentially applying a rotation or a translation of $\RR^2$, we may assume that $p_i$ is the origin and $p_j$ is on the $x$-axis. 
By potentially applying a uniform scaling of $\RR^2$, we may further write 
$p_i = (0,0)$, $p_k = (1,0)$, and $p_s = (a,b)$.
We denote the coordinates of this plane as $(x,y)$. 
According to a different set of coordinates $(u,v)$, we may write $p_j=(0,0)$, $p_\ell = (w,0)$, and $p_t=(c,d)$, where $w\neq 0,1$.

With the above in mind, satisfying $(q_1,q_2)\in C_{i,j}\cap C_{k,l} \cap C_{q,r}$ leads to the system
 \begin{align}
   q_{1,x}^2+q_{1,y}^2 & = q_{2,x}^2+q_{2,y}^2, \label{eq:diffdist:first} \\[2mm]
  (q_{1,x}-1)^2 + q_{1,y}^2  & = (q_{2,x}-w)^2 + q_{2,y}^2, \label{eq:diffdist:second} \\[2mm]
  (q_{1,x}-a)^2 + (q_{1,y}-b)^2 & = (q_{2,x}-c)^2 + (q_{2,y}-d)^2. \label{eq:diffdist:third}
 \end{align}

Subtracting \eqref{eq:diffdist:first} from \eqref{eq:diffdist:second} gives 
 \begin{equation}\label{eq:diffdist:xcoord} 
  q_{1,x} = q_{2,x}w + \frac{1}{2}(1-w^2).
 \end{equation}
Subtracting \eqref{eq:diffdist:first} from \eqref{eq:diffdist:third}
 and plugging \eqref{eq:diffdist:xcoord} into the result leads to
 \begin{equation}\label{eq:diffdist:ycoord}
  bq_{1,y} = (c - aw) q_{2,x} +dq_{2,y} + \frac{1}{2}(a^2+b^2 -c^2 - d^2 +aw^2 -a).
 \end{equation}
Next, we plug \eqref{eq:diffdist:xcoord} and \eqref{eq:diffdist:ycoord} 
 into \eqref{eq:diffdist:first}, to obtain that
 \begin{equation}\label{eq:diffdist:quad}
  (b^2w^2 + (c-aw)^2 -b ^2) q_{2,x}^2 + (d^2-b^2) q_{2,y}^2 + 2d(c-aw)q_{2,x}q_{2,y} + h(q_{2,x},q_{2,y}) = 0.
 \end{equation}
Here, $h(q_{2,x},q_{2,y})$ is a linear function. 

If the left side of \eqref{eq:diffdist:quad} is not identically zero, then $q_2$ lies on a degree 2 curve.
By the assumption on $\gamma_2$, Lemma \ref{le:pfBez} states that the above curve intersects $\gamma_2$ in $O_{\alpha,\beta,r}(1)$ points.
That is, there are $O_{\alpha,\beta,r}(1)$ choices of $q_2$. 
When fixing $q_2$, equations \eqref{eq:diffdist:xcoord} and \eqref{eq:diffdist:ycoord} uniquely determine $q_1$.
Thus, $|C_{i,j}\cap C_{k,\ell} \cap C_{s,t}| =O_{\alpha,\beta,r}(1)$, as required.

It remains to consider the case where the left side of \eqref{eq:diffdist:quad} is identically zero.  
For the coefficients of the quadratic terms in \eqref{eq:diffdist:quad} to be zero, we must have that $b=d=0$ and $c=aw$.
Plugging these into \eqref{eq:diffdist:third} and subtracting from \eqref{eq:diffdist:second} gives
 \[  (2q_{1,x} - a - 1)(a - 1) = (2q_{2,x} - aw - w)(a - 1)w. \]
Applying \eqref{eq:diffdist:xcoord} on the above and rearranging leads to 
 \begin{equation*} 
      a(a - 1)(w + 1)(w - 1) = 0.
 \end{equation*}
 
For this equation to hold, at least one of the following must occur: $a=0$, $a=1$, $w=1$, or $w=-1$.
We recall that $p_i=(0,0)$, $p_k=(1,0)$, $p_s=(a,0)$. 
By the other coordinate system, $p_j=(0,0)$, $p_\ell=(w,0)$, and $p_t=(aw,0)$.
We consider each of the above cases:
\begin{itemize}[itemsep=1pt,topsep=1pt]
\item When $a=0$, we have that $|p_ip_s|=0=|p_jp_t|$, which contradicts the assumption $|p_ip_s|\neq |p_jp_t|$.
\item When $w=\pm 1$, we have that $|p_ip_k|=1=|p_jp_\ell|$, which contradicts the assumption $|p_ip_k|\neq|p_jp_\ell|$.
\item When $a=1$, we have that $|p_kp_s|=0=|p_\ell p_t|$, which contradicts the assumption $|p_kp_s|\neq|p_\ell p_t|$.
\end{itemize}
We conclude that the left side of \eqref{eq:diffdist:quad} cannot be identically zero, which completes the proof.
\end{proof}

The two following lemmas are proved in \cite{PachDZ17}. 
These lemmas do not involve Pfaffian functions, since they consider an algebraic $\gamma_2$ and do not involve $\gamma_1$.

\begin{lemma}\label{le:conic}
Consider $p_i,p_j,p_k,p_\ell,p_s,p_t\in \pts_1$ such that 
 $|p_i p_k| \neq |p_j p_\ell|$, 
 $|p_i p_s| \neq |p_j p_t|$, and 
 $|p_k p_s| \neq |p_\ell p_t|$.
If $\gamma_2$ is an arc of an algebraic curve of degree 2 then
\[ |C_{i,j}\cap C_{k,l} \cap C_{q,r}| \le 4. \]
\end{lemma} 
   
\begin{lemma}\label{le:line}
 Suppose that $|p_ip_k| \neq |p_jp_\ell|$.
 If $\gamma_2$ is a line, then $$|C_{i,j}\cap C_{k,l}| \leq 4.$$
\end{lemma}

Finally, we combine the four lemmas above, to prove part (b) of Lemma \ref{le:CijInt}.

\begin{proof}[Proof of Lemma \ref{le:CijInt}-\ref{cij:no-three-infinite}.]
We say that a symmetry $T$ of $\gamma_2$ \emph{respects} $C_{i,j}$ if $T(p_i)=p_j$.
By Lemma \ref{le:samedist}, if $C_{i,j}$ and $C_{k,\ell}$ have infinite intersection and $|p_ip_k|=|p_jp_\ell|$, then there is a symmetry that respects both $C_{i,j}$ and $C_{k,\ell}$.

We first consider the case where $\gamma_1$ is neither an arc of a line nor an arc of a circle.
In this case, Lemma \ref{lem:syms} states that $\gamma_1$ has $O_{\alpha,\beta,r}(1)$ symmetries.
A fixed symmetry $T$ sends each $p_i$ to a distinct $p_j$.
Thus, $T$ respects $O(m')$ curves $C_{i,j}$.
We conclude that $O_{\alpha,\beta,r}(m')$ curves of $\curves$ are respected by a symmetry of $\gamma_2$.
Let $\curves_0$ be the set of curves of $\curves$ that are respected by such a symmetry.

Next, we assume that $\gamma_2$ is an arc of a circle.
In this case, $\gamma_2$ has infinitely many symmetries.
By the pruning at the beginning of the proof of Theorem \ref{th:twocurves}, every circle concentric to $\gamma_2$ contains at most one point of $\pts_1$. 
Thus, for $p_i,p_j \in \pts_1$ with $i\neq j$, no symmetry of $\gamma_2$ respects $C_{i,j}$.
Since the center of the circle of $\gamma_2$ is not in $\pts_1$, no symmetry of $\gamma_2$ respects $C_{i,i}$.
We may thus set $\curves_0=\emptyset$. 
The case where $\gamma_2$ is a line segment is symmetric.

Combining Lemma \ref{le:samedist} with the above, if $C_{i,j}, C_{k,\ell}\in \curves\setminus \curves_0$ have an infinite intersection, then $|p_ip_k|\neq|p_jp_\ell|$.
Then, Lemmas \ref{lem:diffdist}, \ref{le:conic}, and \ref{le:line}, imply that no three curves of $\curves\backslash \curves_0$ have an infinite intersection.
\end{proof}

\bibliographystyle{plain}
\bibliography{Pfaffian}

\begin{thebibliography}{10}

\bibitem{ALPSV24}
Toby Aldape, Jingyi Liu, Gregory Pylypovych, Adam Sheffer, and Minh-Quan Vo.
\newblock Distinct distances in $\mathbb{R}^3$ between quadratic and orthogonal
  curves.
\newblock {\em European J. Combin.}, 120, 2024.

\bibitem{AC24}
Nuno Arala and Sam Chow.
\newblock Incidence geometry and polynomial expansion over finite fields, 2024.
\newblock \url{https://arxiv.org/abs/2408.02518}.

\bibitem{Balsera23}
Alexander Balsera.
\newblock Incidences with {P}faffian curves and functions, 2023.
\newblock \url{https://arxiv.org/abs/2311.05517}.

\bibitem{BR18}
Saugata Basu and Orit~E. Raz.
\newblock An o-minimal {S}zemer\'edi-{T}rotter theorem.
\newblock {\em Q. J. Math.}, 69(1):223--239, 2018.

\bibitem{BKT04}
Jean Bourgain, Netz Katz, and Terence Tao.
\newblock A sum-product estimate in finite fields, and applications.
\newblock {\em Geom. Funct. Anal.}, 14(1):27--57, 2004.

\bibitem{BMP05}
Peter Brass, William Moser, and J\'anos Pach.
\newblock {\em Research problems in discrete geometry}.
\newblock Springer, New York, 2005.

\bibitem{CGS20}
Artem Chernikov, David Galvin, and Sergei Starchenko.
\newblock Cutting lemma and {Z}arankiewicz's problem in distal structures.
\newblock {\em Selecta Math. (N.S.)}, 26(2):Paper No. 25, 27, 2020.

\bibitem{CPS24}
Artem Chernikov, Ya'acov Peterzil, and Sergei Starchenko.
\newblock Model-theoretic {E}lekes-{S}zab\'o{} for stable and o-minimal
  hypergraphs.
\newblock {\em Duke Math. J.}, 173(3):419--512, 2024.

\bibitem{CS18}
Artem Chernikov and Sergei Starchenko.
\newblock Regularity lemma for distal structures.
\newblock {\em J. Eur. Math. Soc. (JEMS)}, 20(10):2437--2466, 2018.

\bibitem{CST21}
Artem Chernikov, Sergei Starchenko, and Margaret E.~M. Thomas.
\newblock Ramsey growth in some {NIP} structures.
\newblock {\em J. Inst. Math. Jussieu}, 20(1):1--29, 2021.

\bibitem{Conrad}
Keith Conrad.
\newblock Isometries of the plane and linear algebra.
\newblock {\em Group Theory Blurbs, Expository Papers by Conrad}.
\newblock
  \url{https://kconrad.math.uconn.edu/blurbs/grouptheory/isometryR2.pdf}.

\bibitem{Erdos46}
Paul Erd{\H{o}}s.
\newblock On sets of distances of {$n$} points.
\newblock {\em Amer. Math. Monthly}, 53:248--250, 1946.

\bibitem{FPSSZ17}
Jacob Fox, J\'anos Pach, Adam Sheffer, Andrew Suk, and Joshua Zahl.
\newblock A semi-algebraic version of {Z}arankiewicz's problem.
\newblock {\em J. Eur. Math. Soc. (JEMS)}, 19(6):1785--1810, 2017.

\bibitem{GV95}
Andrei Gabrielov and Nicolai Vorobjov.
\newblock Complexity of stratifications of semi-{P}faffian sets.
\newblock {\em Discrete Comput. Geom.}, 14(1):71--91, 1995.

\bibitem{GV04}
Andrei Gabrielov and Nicolai Vorobjov.
\newblock Complexity of computations with {P}faffian and {N}oetherian
  functions.
\newblock In {\em Normal forms, bifurcations and finiteness problems in
  differential equations}, volume 137 of {\em NATO Sci. Ser. II Math. Phys.
  Chem.}, pages 211--250. Kluwer Acad. Publ., Dordrecht, 2004.

\bibitem{GK15}
Larry Guth and Nets~Hawk Katz.
\newblock On the {E}rd{\H{o}}s distinct distances problem in the plane.
\newblock {\em Ann. of Math. (2)}, 181(1):155--190, 2015.

\bibitem{Khovanskii91}
Askold~G. Khovanski{\u\i}.
\newblock {\em Fewnomials}, volume~88 of {\em Translations of Mathematical
  Monographs}.
\newblock American Mathematical Society, Providence, RI, 1991.
\newblock Translated from the Russian by Smilka Zdravkovska.

\bibitem{KP02}
Steven~G. Krantz and Harold~R. Parks.
\newblock {\em A primer of real analytic functions}.
\newblock Birkh\"auser Advanced Texts: Basler Lehrb\"ucher. Birkh\"auser
  Boston, Inc., Boston, MA, second edition, 2002.

\bibitem{LNV24}
Martin Lotz, Abhiram Natarajan, and Nicolai Vorobjov.
\newblock Partitioning theorems for sets of semi-{P}faffian sets, with
  applications, 2024.
\newblock \url{https://arxiv.org/abs/2412.02961}.

\bibitem{Osgood16}
William~F. Osgood.
\newblock On functions of several complex variables.
\newblock {\em Trans. Amer. Math. Soc.}, 17(1):1--8, 1916.

\bibitem{PachDZ17}
J\'anos Pach and Frank de~Zeeuw.
\newblock Distinct distances on algebraic curves in the plane.
\newblock {\em Combin. Probab. Comput.}, 26(1):99--117, 2017.

\bibitem{PS98}
J\'anos Pach and Micha Sharir.
\newblock On the number of incidences between points and curves.
\newblock {\em Combin. Probab. Comput.}, 7(1):121--127, 1998.

\bibitem{RSdZ16}
Orit~E. Raz, Micha Sharir, and Frank De~Zeeuw.
\newblock Polynomials vanishing on {C}artesian products: the
  {E}lekes-{S}zab\'o{} theorem revisited.
\newblock {\em Duke Math. J.}, 165(18):3517--3566, 2016.

\bibitem{RSS16}
Orit~E. Raz, Micha Sharir, and J\'ozsef Solymosi.
\newblock Polynomials vanishing on grids: the {E}lekes-{R}\'onyai problem
  revisited.
\newblock {\em Amer. J. Math.}, 138(4):1029--1065, 2016.

\bibitem{SSS13}
Micha Sharir, Adam Sheffer, and J\'ozsef Solymosi.
\newblock Distinct distances on two lines.
\newblock {\em J. Combin. Theory Ser. A}, 120(7):1732--1736, 2013.

\bibitem{Sheffer14}
Adam Sheffer.
\newblock Distinct distances: Open problems and current bounds, 2018.
\newblock \url{https://arxiv.org/abs/1406.1949}.

\bibitem{Sheffer22}
Adam Sheffer.
\newblock {\em Polynomial methods and incidence theory}, volume 197 of {\em
  Cambridge Studies in Advanced Mathematics}.
\newblock Cambridge University Press, Cambridge, 2022.

\bibitem{SV08}
J\'ozsef Solymosi and Van~H. Vu.
\newblock Near optimal bounds for the erd{\H{o}}s distinct distances problem in
  high dimensions.
\newblock {\em Combinatorica}, 28(1):113--125, 2008.

\bibitem{SolyZahl24}
Jozsef Solymosi and Joshua Zahl.
\newblock Improved {E}lekes-{S}zab\'o{} type estimates using proximity.
\newblock {\em J. Combin. Theory Ser. A}, 201:Paper No. 105813, 9, 2024.

\bibitem{Sp99}
Patrick Speissegger.
\newblock The {P}faffian closure of an o-minimal structure.
\newblock {\em J. Reine Angew. Math.}, 508:189--211, 1999.

\bibitem{Stein89}
Yosef Stein.
\newblock The total reducibility order of a polynomial in two variables.
\newblock {\em Israel J. Math.}, 68(1):109--122, 1989.

\bibitem{TidorYu24}
Jonathan Tidor and Hung-Hsun~Hans Yu.
\newblock Multilevel polynomial partitioning and semialgebraic hypergraphs:
  regularity, {T}ur\'an, and {Z}arankiewicz results, 2024.
\newblock \url{https://arxiv.org/abs/2407.20221}.

\bibitem{Dries98}
Lou van~den Dries.
\newblock {\em Tame topology and o-minimal structures}, volume 248 of {\em
  London Mathematical Society Lecture Note Series}.
\newblock Cambridge University Press, Cambridge, 1998.

\end{thebibliography}

\end{document}